\documentclass[10pt,a4paper]{article}
\usepackage[utf8]{inputenc}
\usepackage[english]{babel}
\usepackage{hyperref}
\usepackage{amsmath,amsthm,enumerate}
\usepackage{amssymb}
\usepackage{tikz}
\usetikzlibrary{shapes}

\tikzstyle{main node}=[circle, draw,
                        inner sep=0pt, minimum width=15pt]
\tikzstyle{code node}=[circle, draw, fill=lightgray,
                        inner sep=0pt, minimum width=15pt]
\tikzstyle{shift node}=[circle, draw, fill=red,
                        inner sep=0pt, minimum width=15pt]
\tikzstyle{square main node}=[shape=rectangle, draw,
                        inner sep=7pt, minimum width=15pt]
\tikzstyle{square main node2}=[shape=rectangle, draw,
                        inner sep=5pt, minimum width=10pt]
\tikzstyle{square code node}=[shape=rectangle, draw, fill=lightgray,
                        inner sep=7pt, minimum width=15pt]
\tikzstyle{square code node2}=[shape=rectangle, draw, fill=lightgray,
                        inner sep=5pt, minimum width=10pt]
\tikzstyle{small node}=[circle, draw,
                        inner sep=0pt, minimum width=6pt]
\usepackage{subfigure}

\usepackage[hmargin=2.5cm,vmargin=3cm]{geometry}
\usepackage[color=green!30]{todonotes}

\usepackage{perpage} 
\MakePerPage{footnote} 

\newtheorem{theorem}{Theorem}

\newtheorem{proposition}[theorem]{Proposition}
\newtheorem{lemma}[theorem]{Lemma}

\newtheorem{corollary}[theorem]{Corollary}

\newcommand{\ID}{\gamma^{\text{\tiny{ID}}}}

\newcommand{\IDt}{\gamma_t^{\text{\tiny{ID}}}}

\begin{document}

\title{Revisiting and improving upper bounds for identifying codes\footnote{A shorter version of the article has been presented at the 11th International Colloquium on Graph Theory and combinatorics, Montpellier, 2022.}}
\author{Florent Foucaud\footnote{\noindent Université Clermont-Auvergne, CNRS, Mines de Saint-Étienne, Clermont-Auvergne-INP, LIMOS, 63000 Clermont-Ferrand, France. Research supported by the French government IDEX-ISITE initiative 16-IDEX-0001 (CAP 20-25) and by the ANR project GRALMECO (ANR-21-CE48-0004).}~\footnote{Univ. Orléans, INSA Centre Val de Loire, LIFO EA 4022, F-45067 Orléans Cedex 2, France.}
\and Tuomo Lehtil\"a\footnote{\noindent Univ Lyon, Universit\'e Claude Bernard, CNRS, LIRIS - UMR 5205, F69622, France. Research supported by the Finnish Cultural Foundation.}~\footnote{Department of mathematics and statistics, University of Turku, Finland.  Research supported by the Academy of Finland grant 338797.}
}

\maketitle

\begin{abstract}
An identifying code $C$ of a graph $G$ is a dominating set of $G$ such that any two distinct vertices of $G$ have distinct closed neighbourhoods within $C$. These codes have been widely studied for over two decades.
We give an improvement over all the best known upper bounds, some of which have stood for over 20 years, for identifying codes in trees, proving the upper bound of $(n+\ell)/2$, where $n$ is the order and $\ell$ is the number of leaves (pendant vertices) of the graph. In addition to being an improvement in size, the new upper bound is also an improvement in generality, as it actually holds for bipartite graphs having no twins (pairs of vertices with the same closed or open neighbourhood) of degree~2 or greater. We also show that the bound is tight for an infinite class of graphs and that there are several structurally different families of trees attaining the bound. We then use our bound to derive a tight upper bound of $2n/3$ for twin-free bipartite graphs of order $n$, and characterize the extremal examples, as 2-corona graphs of bipartite graphs. This is best possible, as there exist twin-free graphs, and trees with twins, that need $n-1$ vertices in any of their identifying codes. We also generalize the existing upper bound of $5n/7$ for graphs of order $n$ and girth at least~5 when there are no leaves, to the upper bound $\frac{5n+2\ell}{7}$ when leaves are allowed. This is tight for the $7$-cycle $C_7$ and for all stars.
\end{abstract}

\section{Introduction}

An identifying code of a graph is a subset of its vertices that allows to distinguish all pairs of vertices by means of their neighbourhoods within the identifying code. This concept is related to other similar notions that deal with domination-based identification of the vertices/edges of a graph or hypergraph, such as locating-dominating sets~\cite{S87}, separating systems~\cite{BS07,B72}, or test covers~\cite{MS85}, to name a few. These types of problems have natural applications in fault-detection in networks~\cite{KCL98,UTS04}, biological diagnosis~\cite{MS85} or machine learning~\cite{BGL19}.

A lot of the research in the area has been dedicated to understanding the behaviour of these types of problems for graphs of specific graph classes, by proving lower and upper bounds on the smallest cardinality of a solution. One of the simplest classes of graphs to consider is the one of trees, and indeed a large number of papers in the area consider identifying codes of trees, see for example~\cite{A10,BCHL05,BCMMS07,CS11,GGNUV01,HHH06,JR18,NLG16,RRM19,S87}. We improve and generalize some of these results. As some of the best known bounds (see Theorem \ref{2lTreeBound}), which are tight for some trees, have been around for more than twenty years, it is quite surprising that we manage to improve the upper bound for the smallest size of identifying codes of trees by a notable amount, with a quite simple proof. Moreover, our bounds do not only hold for trees, but also for larger classes of graphs, in particular, bipartite graphs without twins of degree at least~2. We also apply the new bound to get a new tight bound for graphs of girth at least~5.

\paragraph{Notations and definitions.} We consider connected finite undirected graphs on at least three vertices. Let us first define some basic notations. A vertex $u\in V(G)$ is said to be a \textit{leaf}, that is, a pendant vertex, if it has degree~1. A vertex $v\in V(G)$ is said to be a \textit{support vertex} if it is adjacent to a leaf.
We denote by $L(G)$ the set of leaves and by $S(G)$ the set of support vertices in graph $G$. Moreover, we denote the number of leaves and support vertices by $\ell(G)=|L(G)|$ and $s(G)=|S(G)|$, respectively. A graph is \textit{bipartite} if it does not contain any odd cycles and it has \textit{girth} $g$ if the length of its smallest cycle is~$g$. The \emph{$2$-corona} $H\circ_2$ of a graph $H$ (defined in~\cite[Section 1.3]{bookTD}) is the graph of order~$3|V(H)|$ obtained from $H$ by adding a vertex-disjoint copy of a path $P_2$ for each vertex $v$ of $H$ and adding an edge joining $v$ to one end of the added path. We define the \emph{$k$-corona} $H\circ_k$ of $H$ in an analogous way with $P_k$, for any $k\geq1$.

We denote by $N(v)\subseteq V(G)$ the \textit{open neighbourhood} of vertex $v$ and by $N[v]=N(v)\cup\{v\}$, its \textit{closed neighbourhood}. If $C$ is a set of vertices, or a \textit{code}, and $v$, a vertex, we denote the intersection between $N[v]$ and code $C$ by the $I$-set of $v$, $I(C;v)=N[v]\cap C$. When code $C$ is clear from the context, we use $I(v)$. Identifying codes were defined over twenty years ago in~\cite{KCL98} by Karpovsky et al. and since then they have been studied in a large number of articles, see~\cite{biblio} for an online bibliography. A set $C\subseteq V(G)$ is called an \textit{identifying code} of $G$ if for each pair of distinct vertices $u,v\in V(G)$, we have that (i) they are \textit{covered}/\textit{dominated}, that is, $I(u)\neq \emptyset$ and $I(v)\neq \emptyset$ and (ii) they are \textit{distinguished}/\textit{separated}, that is, their $I$-sets are distinct, that is, $$I(u)\neq I(v).$$ 
The vertices of the code are called \emph{codewords}. Two distinct vertices are \emph{open twins} if their open neighbourhoods are the same, and \emph{closed twins} if their closed neighbourhoods are the same. A graph admits an identifying code if and only if it has no pair of closed twins~\cite{KCL98}; in that case we say the graph is \emph{identifiable}. Note that any connected bipartite (in fact, triangle-free) graph is identifiable, with the exception of the complete graph of order~2. We say that a graph is \textit{twin-free} if it contains neither open nor closed twins. Twins are important for identifying codes, indeed closed twins cannot be identified, and for any set of mutually open twins, at most one can be absent from the identifying code.
For an identifiable graph $G$, we denote by $\ID(G)$ the smallest size of an identifying code of $G$.

Identifying codes and related concepts have been extensively studied for trees;
in particular, lower and upper bounds involving the number of leaves and support vertices have been proposed. Among these, the following two (due to Gimbel et al. and Rahbani et al.) are the currently best known upper bounds.

\begin{theorem}[\protect{\cite[Theorem 15]{GGNUV01}}]\label{2lTreeBound}
Let $T$ be a tree on $n\geq3$ vertices. Then $\ID(T)\leq \frac{n+2\ell(T)-2}{2}$. 
\end{theorem}

\begin{theorem}[\protect{\cite[Theorem 11]{RRM19}}]\label{3/5TreeBound}
Let $T$ be a tree on $n\geq3$ vertices. Then $\ID(T)\leq \frac{3n+2\ell(T)-1}{5}$. Equality holds if and only if $T=P_4$.
\end{theorem}

The bound in Theorem~\ref{2lTreeBound} is better when the tree has few leaves, while the bound in Theorem~\ref{3/5TreeBound} is better when there are many leaves. Both bounds are tight for the 4-vertex path $P_4$, for which $\ID(P_4)=3$. Moreover, Theorem~\ref{2lTreeBound} is tight for any path on at least four vertices since $\ID(P_n)=\left\lceil\frac{n+1}{2}\right\rceil$ as proved in~\cite{BCHL04,GGNUV01}. However, we will see that tightness only holds for this case and on some trees of odd order with three leaves.

The following bound for graphs of girth at least~5 with no leaves was also proved by Balbuena et al.

\begin{theorem}[\protect{\cite[Theorem 13]{BFH15}}]\label{The:ID5/7delta}
Let $G$ be a graph of order $n$ and girth at least $5$ with minimum degree $\delta(G)\geq2$. Then $\ID(G)\leq 5n/7$.
\end{theorem}

\paragraph{Our results.}
Inspired by the aforementioned results from the literature, we present improved (and tight) upper bounds. Our bounds not only improve on the known results for trees, but also hold for a larger class of graphs: bipartite graphs which do not have any twins of degree~2 or greater. Observe that this class of graphs contains, for example, the class of $C_4$-free bipartite graphs.

In particular, we show in Theorem~\ref{TheShift} that the bound $\ID(G)\leq\frac{n+\ell(G)}{2}$ holds for every bipartite graph $G$ of order $n$ which does not have any twins of degree~2 or greater. This bound is never larger than either of the bounds of Theorems~\ref{2lTreeBound} or~\ref{3/5TreeBound}. In fact, Theorem~\ref{2lTreeBound} is only tight when $\ell(T)\leq 3$. (Otherwise, one can check that our bound is smaller, indeed then we have $\lfloor(n+\ell(T))/2\rfloor<\lfloor(n+2\ell(T)-2)/2\rfloor$.) The bound of Theorem~\ref{3/5TreeBound} can be modified into the equivalent form $\ID(T)\leq (3n-3\ell(T)-1)/5+\ell(T)$. If we similarly modify the bound of Theorem~\ref{TheShift} to $\ID(T)\leq (n-\ell(T))/2+\ell(T)$, then we can clearly observe the improvement provided by our bound. 
Moreover, as opposed to the existing bounds, our bound is tight for a very rich class of graphs, in particular, for many trees: paths, stars, and more complicated examples that will be described in Section~\ref{Sec:Remarks}. Our proof is also rather simple.

We also extend the bound $\ID(G)\leq n - s(G)$ (even holding for identifying codes that are also total dominating sets), which was known to hold for trees~\cite{HHH06}, to a class that includes all triangle-free graphs. We also show that the slightly larger bound $\ID(G)\leq n-s(G)+1$ holds for all identifiable graphs. Both bounds are tight.

We then apply two of the above upper bounds to \emph{twin-free} bipartite graphs, showing that for such a graph $G$ of order $n$, we always have $\ID(G)\leq 2n/3$ (unless $G$ is the path $P_4$). Moreover, we characterize those graphs reaching this bound, as the 2-coronas of bipartite graphs.

Finally, we extend the bound of Theorem~\ref{The:ID5/7delta} to all graphs of girth at least~5, showing that for such a graph $G$ of order $n$, we have $\ID(G)\leq\frac{5n+2\ell(G)}{7}$. This bound is tight for all stars and for the cycle $C_7$. We also present a new infinite family of graphs of girth at least $5$ which has largest known ratio $\ID(G)/n$ for graphs with $\delta(G)\geq2$.

We present our upper bounds for identifying codes in bipartite graphs in Section~\ref{SecBounds}. The application to twin-free graphs is presented in Section~\ref{sec:twinfree}. The bound for graphs of girth at least~5 is proved in Section~\ref{sec:girth5}.

We conclude in Section~\ref{sec:conclu}.

\paragraph{Further related work.}
For a tree $T$ of order $n\geq 4$, the lower bounds $\ID(T)\geq\frac{3(n-1)}{7}$~\cite{BCHL05,GGNUV01}, $\ID(T)\geq\frac{2n-s(T)+3}{4}$~\cite{RRM19} and $\ID(T)\geq\frac{3n+\ell(T)-s(T)+1}{7}$~\cite{BCMMS07} have been proved.

We note that a polynomial-time algorithm to compute $\ID(T)$ for a tree $T$ has been provided in~\cite{A10}, however the problem is hard to approximate within any sub-logarithmic factor even for bipartite graphs with no 4-cycles~\cite{BLLPT15} (bipartite graphs wihout $4$-cycles do not have any twins of degree~2 or greater). Identifying codes in graphs of girth at least~5 were also considered in~\cite{BFH15,FP11}.

Many upper bounds for trees similar to those of Theorems~\ref{2lTreeBound} and \ref{3/5TreeBound} have been obtained for related graph parameters. 
In particular, tight upper bounds on trees for identifying codes that are also total dominating sets have been considered in~\cite{JR18,NLG16}. Similar results have also have been proved for the related locating-dominating sets (where one only needs to distinguish pairs of vertices that are not part of the solution set)~\cite{BCMMS07,BDLP21,JR18} and their total dominating variant locating-total dominating sets~\cite{CS11}.

Bounds for twin-free graphs have been studied for locating-dominating sets. It was proved in~\cite{GGM14} that every twin-free bipartite graph and every twin-free graph with no $4$-cycles has a locating-dominating set of size at most $n/2$; the bound is tight for infinitely many trees, which are characterized in~\cite{Heia}.

\section{Two upper bounds}\label{SecBounds}

In this section, we present our main result in Theorem~\ref{TheShift}. However, first we give some upper bounds which are useful when the number of support vertices in $G$ is large. 

\subsection{A first pair of upper bounds using the number of support vertices}

Our first lemma holds for identifying codes that are also \emph{total dominating sets}, that is, every vertex of the graph has a neighbour in the dominating set~\cite{bookTD}. A graph admits no total dominating set only if it has isolated vertices. For an identifiable graph $G$ with no isolated vertices, we denote by $\IDt(G)$ the smallest size of a total dominating identifying code of $G$. Total dominating identifying codes have usually been called \emph{differentiating-total dominating sets} in the literature, see for example~\cite{GGNUV01,HHH06,JR18,NLG16}.

The following lemma has previously been proven for trees in~\cite{HHH06}. We extend it to a larger class which contains, for example, all triangle-free graphs (note that for every triangle-free graph $G$, $G-L(G)$ is identifiable, unless $G-L(G)$ contains a component isomorphic to $P_2$).

\begin{lemma}\label{LemSupportBip}
Let $G$ be a connected graph on $n\geq4$ vertices that is not the path $P_4$, such that $G-L(G)$ is identifiable or $G$ is triangle-free.
Then $$\IDt(G)\leq n-s(G).$$
\end{lemma}
\begin{proof}
Let us choose for each support vertex $u\in S(G)$ exactly one adjacent leaf $v\in L(G)$ and say that these vertices form the set $C'$. Hence, $|C'|=s(G)$. Next, we form the code $C=V(G)\setminus C'$. We show that $C$ is a total dominating identifying code in $G$. We have $|C|=n-s(G)$. 

Observe that each non-codeword $v$ is a leaf. Moreover, if $u,v\in C'$ and $I(C;v)=I(C;u)=\{w\}$, then $w\in S(G)$ but we have chosen two vertices adjacent to $w$ into $C'$, a contradiction. Since $C'\subseteq L(G)$, $G[C]$ is a connected induced subgraph of $G$. Moreover, as $n\geq4$, we have $|I(C;c)|\geq2$ for each codeword $c\in C$. Thus, codewords and non-codewords have different $I$-sets. Furthermore, we have $I(c)\neq I(c')$ for two codewords $c\neq c'$ since $|V(G[C])|\geq 3$ (as $G$ is not $P_4$) and there are no triangles in $G[C]$ or $G[C]$ is identifiable and hence, each closed neighbourhood is unique in $G[C]$. Finally, $C$ is total dominating since $G[C]$ is connected.
\end{proof}

Lemma~\ref{LemSupportBip} is tight for example for 3-coronas of graphs (but for these graphs the regular identifying code number is smaller). It is also tight for both identifying codes and total dominating identifying codes, by considering for example the 1-corona of any triangle-free graph of order at least~3, or any star of order at least~4.

Moreover, we require the stated restrictions in the claim. For example, for the 1-corona $K_m\circ_1$ of a complete graph of order $m\geq3$ (for which $K_m\circ_1 - L(K_m\circ_1)$ is far from identifiable), we have $\IDt(K_m\circ_1)=2m-1$ and not $2m-s(K_m\circ_1)=m$. Indeed, all vertices of the $m$-clique need to be in the code to totally dominate the leaves. Moreover, for any two vertices of the clique, one of their two leaf neighbours needs to be in the code to identify them, and hence $\IDt(K_m\circ_1)\geq 2m-1$. Moreover, $\IDt(K_m\circ_1)\leq 2m-1$ by considering the whole vertex set except one leaf as a code.

The same example shows that the statement of Lemma~\ref{LemSupportBip} is also not true for identifying codes on general graphs, as $\ID(K_m\circ_1)=m+1$ and not $2m-s(K_m\circ_1)=m$. Indeed, as above, we need at least $m-1$ leaf codewords, to distinguish the vertices of the $m$-clique. Moreover, to distinguish leaves from their neighbour support vertex, we need at least two codewords inside the $m$-clique. This implies $\ID(K_m\circ_1)\geq m+1$. Conversely, consider $L(G)$ together with any two vertices in the $m$-clique, and remove a leaf neighbour of one of the clique code vertices: this set forms an identifying code of size $m+1$.

In the following theorem, we show that the previous construction is actually the worst case for identifying codes: in this case, a very similar upper bound as the one of Lemma~\ref{LemSupportBip} is true.

\begin{theorem}
Let $G$ be a connected identifiable graph on $n\geq3$ vertices. Then $$\ID(G)\leq n-s(G)+1.$$
\end{theorem}
\begin{proof}
When $s(G)\leq 2$, the claim is clear since $\ID(G)\leq n-1$ for any connected identifiable graph on at least three vertices~\cite{GM07}. Hence, we may assume that $s(G)\geq3$ which implies $n\geq 2s(G)\geq6$, and so the claim is true when $n\leq 5$. Let us then prove the claim by induction on $n$. We thus assume that for any graph $G'$ of order $n'\leq n-1$, we have $\ID(G')\leq n'-s(G')+1$ and wish to prove it for graphs of order $n$. We know that the claim is true for $s(G)\leq 2$, thus we assume also that $s(G)\geq3$. If $G'=G-L(G)$ is identifiable, then $\ID(G)\leq n-s(G)$ by Lemma~\ref{LemSupportBip}. Thus, we may assume that $N_{G'}[u]=N_{G'}[v]$ for some distinct vertices $u,v$ in $V(G')$. Since $G$ is identifiable, we have $u$ or $v\in S(G)$. Assume that $u\in S(G)$ and let $L_u=N(u)\cap L(G)$. Let us form graph $G_u=G-u-L_u$. 

Observe that $G_u$ is connected since $N_{G'}[v]=N_{G'}[u]$. Moreover, since $s(G)\geq3$ and $G$ is identifiable, also $G_u$ is identifiable. Indeed, if $N_{G_u}[x]=N_{G_u}[y]$ for some vertices $y$ and $x$, then $u$ separates them in $G$, let us say, $u\in N_G[x]\setminus  N_G[y]$. However, we have $N_{G'}[u]=N_{G'}[v]$. Hence, $v\in N_{G_u}[x]\setminus  N_{G_u}[y]$, a contradiction, or $v=x$ and $y$ is a leaf in $G$ which is adjacent to $v$ but now $N_{G_u}[v]=N_{G_u}[y]$ and $G_u$ is a path on two vertices which is impossible since then $s(G)=2$. Notice that we have $s(G_u)\geq s(G)-1\geq2$, since we removed at most one support vertex. Since $s(G_u)\geq2$, $G_u$ has order at least~4.

By the induction hypothesis, $G_u$ has an identifying code of cardinality at most $(n-1-|L_u|)-s(G_u)+1\leq n-|L_u|-s(G)+1$. Let $C_u$ be an identifying code with such cardinality. Let us first consider the case with $|L_u|\geq2$. Now, we may take $C=C_u\cup L_u$ and it is an identifying code of $G$ with at most $n-s(G)+1$ vertices, and we are done. However, if $|L_u|=1$ and let us say that $u'$ is the leaf in $L_u$, then just by adding $u'$ to $C_u$, we might have $I(C;u)=I(C;u')$. Let us divide this into two cases. First assume that,  $I(C_u;v)\not\subseteq L(G)$, now $I(v)\cap N(u)\neq \emptyset$ and $C_u\cup\{u'\}$ is an identifying code in $G$ of cardinality at most $n-s(G)+1$, as required.  Hence, we may assume that $I(C_u;v)\subseteq L(G)$. In this case we have $|I(C_u;v)|=|N(v)\cap L(G)|\geq2$. Now, to form code $C$, we may shift one codeword from $N(v)\cap L(G)$ to $v$ and add $u$ to the code.  The resulting code is identifying in $G$ and has cardinality at most $n-s(G)+1$. \end{proof}

\subsection{A second upper bound using the number of leaves}

Now, we prove an upper bound for identifying codes in some bipartite graphs based on the number of leaves. In particular, this bound is an improvement for trees. Notice that the graphs in the following theorem contain $C_4$-free bipartite graphs.

\begin{theorem}\label{TheShift}
Let $G$ be a connected bipartite graph on $n\geq3$ vertices which does not have any twins of degree~2 or greater. We have $$\ID(G)\leq\frac{n+\ell(G)}{2}.$$
\end{theorem}
\begin{proof}
Let $G$ be a connected bipartite graph on $n\geq3$ vertices which does not have any twins of degree~2 or greater. Let us fix a non-leaf vertex $x\in V(G)\setminus L(G)$ as the root of the graph. We present the bipartite graph as a layered graph, so that vertex $u$ is in layer $i$ if $d(x,u)=i$. Observe that adjacent vertices have different distances to the root $x$ since $G$ is bipartite. Our goal is to construct two identifying codes, and to show that at least one of them has the claimed cardinality.

Let us first construct code $C'_e$ by choosing $u\in C'_e$ if $d(x,u)$ is even, or if $u\in L(G)$. Next we will shift some codewords to construct an identifying code $C_e$. Observe that if $u\in L(G)$ has an odd distance to root $x$, then there is an adjacent support vertex $v\in S(G)\cap I(C'_e;u)$. Let us first have $C_e=C'_e$. We modify code $C_e$ by shifting some codewords in leaves according to following rule. If $N(v)\cap L(G)=\{u\}$ and $d(x,u)$ is odd, then we remove $u$ from $C_e$ and we add some vertex $v'\in N(v)$ with $d(v',x)=d(v,x)-1$ to $C_e$ (if $v=x$, we instead add some non-leaf vertex adjacent to $x$ to $C_e$).  We illustrate codes $C_e'$ and $C_e$ together with the shift in the left graph of Figure~\ref{TreeExamp}. We next prove that $C_e$ is an identifying code in $G$.

If $w\not\in L(G)$ is in layer $i$ where $i$ is odd, then $w$ has at least two adjacent codewords and $N(w)\subseteq I(C_e;w)$. Thus, $w$ has a unique $I$-set since if $I(w')=I(w)$ for some vertex $w'\neq w$, then $w'$ is a non-leaf in an odd layer and $N(w')\subseteq I(w)$. Thus, $w$ and $w'$ are twins of degree at least~2, a contradiction. Let us then consider the case where $w\in L(G)$ and $d(w,x)$ is odd; assume $u\in S(G)$ is the adjacent support vertex. Now, if $|N(u)\cap L(G)|\geq2$, then $I(w)=\{w,u\}$ and $|I(u)|\geq3$. Thus, $I(w)$ is unique. If $|N(u)\cap L(G)|=1$, then $w\not\in C_e$ due to shifting and $I(w)=\{u\}$. However, $|I(u)|\geq2$ and hence, $I(w)$ is again unique. Let us then consider the case where $d(w,x)$ is even. Now, $w\in C_e$ and hence, if $I(w)=I(w')$, then $w'\in N(w)$. But then $d(w',x)$ is odd and thus, $I(w')$ is also unique by our earlier arguments. Thus, $C_e$ is an identifying code in $G$.

As the second code, we construct $C'_o$ similarly as we constructed $C_e'$, except that we choose vertices in odd layers, that is, we have $u\in C'_o$ if $d(x,u)$ is odd, or if $u\in L(G)$. Then, we again use the shifting to obtain the code $C_o$. This time, we shift some codewords away from some leaves in even layers. Let $u\in L(G)$ be a leaf with $d(u,x)$ even and $v\in S(G)\cap N(u)$. Thus, the distance between vertices $v$ and $x$ is odd and $v\in C'_o$. Moreover, let, again, $v'\in N(v)$ be the vertex adjacent to $v$ with $d(v',x)=d(v,x)-1$. Now, if $N(v)\cap L(G)=\{u\}$, then we remove $u$ from $C_o$ and add $v'$ to $C_o$. Codes $C_o'$ and $C_o$ are illustrated in the right graph of Figure~\ref{TreeExamp} and they can be compared with the codes $C_e'$ and $C_e$. The proof that the code $C_o$ is identifying is similar to the proof for $C_e$.

Thus, we have $\ID(G)\leq \min \{|C_e|,|C_o|\}\leq  \min \{|C'_e|,|C'_o|\}$. Moreover, $C'_e$ and $C'_o$ both contain every leaf and at least one of them contains at most half of the non-leaf vertices. Thus, $\ID(G)\leq\min \{|C'_e|,|C'_o|\}\leq \ell(G)+(n-\ell(G))/2=(n+\ell(G))/2.$
\end{proof}

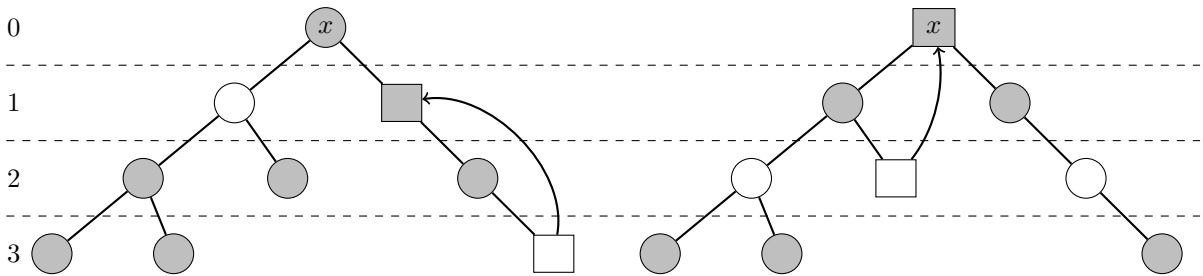
\begin{figure}[!htpb]
\centering
\begin{tikzpicture}

\draw[dashed] (-4.2,-0.5) -- (11.5,-0.5);
\draw[dashed] (-4.2,-1.5) -- (11.5,-1.5);
\draw[dashed] (-4.2,-2.5) -- (11.5,-2.5);

\node[code node](x) at (0,0)    {$ x $};
\node[main node](1) at (-1.2,-1) {};

\node[code node](11) at (-2.4,-2)    {};
\node[code node](12) at (-0.5,-2)  {};
\node[code node](111) at (-3.6,-3) {};
\node[code node](112) at (-2.0,-3) {};

\node[square code node](2) at (1,-1)  {};
\node[code node](21) at (2,-2)  {};
\node[square main node](211) at (3,-3)  {};

\node[] at (-4.1,-0) {$0$};
\node[] at (-4.1,-1) {$1$};
\node[] at (-4.1,-2) {$2$};
\node[] at (-4.1,-3) {$3$};

\path[draw,thick]
    (x) edge node {} (1)
    (1) edge node {} (11)
    (1) edge node {} (12)
    (11) edge node {} (111)
    (11) edge node {} (112)
    (x) edge node {} (2)
    (2) edge node {} (21)
    (21) edge node {} (211)
 (211)[->,bend right=55, looseness=0.9] edge node {} (2)  

    ;
\node[square code node2](Rx) at (8,0)    {$ x $};
\node[code node](R1) at (6.8,-1) {};

\node[main node](R11) at (5.6,-2)    {};
\node[square main node](R12) at (7.5,-2)  {};
\node[code node](R111) at (4.4,-3) {};
\node[code node](R112) at (6.0,-3) {};

\node[code node](R2) at (9,-1)  {};
\node[main node](R21) at (10,-2)  {};
\node[code node](R211) at (11,-3)  {};

\path[draw,thick]
    (Rx) edge node {} (R1)
    (R1) edge node {} (R11)
    (R1) edge node {} (R12)
    (R11) edge node {} (R111)
    (R11) edge node {} (R112)
    (Rx) edge node {} (R2)
    (R2) edge node {} (R21)
    (R21) edge node {} (R211)
    (R12)[->,bend right=25, looseness=0.85] edge node {} (Rx)  
    ;

\end{tikzpicture}\centering
\caption{Gray circle  and white square vertices form the codes $C'_e$ (left) and $C'_o$ (right). Arrows and squared vertices depict the shifts in the forming of identifying codes $C_e$ and $C_o$ which contain gray circled and squared vertices.}\label{TreeExamp}
\end{figure}

\subsection{Remarks and consequences}\label{Sec:Remarks}

Observe that the conditions (being bipartite and not having twins of degree~2 or greater) in the statement of Theorem~\ref{TheShift} are necessary. For example, we have $\ID(C_4)=\ID(C_5)=3$, and when $n\geq 7$ is odd, we have $\ID(C_n)=\lceil\frac{n}{2}\rceil+1$~\cite{GGNUV01}; for a complete bipartite graph $K_{k_1,k_2}$ of order $n$, with $k_1>k_2\geq2$, we have $\ID(K_{k_1,k_2})=n-2$.

The upper bound given by Theorem~\ref{TheShift} is tight for quite a rich class of graphs that includes many structurally different graphs. Those graphs include, for example, any path or even-length cycle with at least six vertices~\cite{BCHL04}, or any star on at least three vertices (for a star $S_n$ of order $n$, we have $\ID(S_n)=n-1$). The bound is also tight for any \emph{spider graph} where the length of each leg is odd (that is, a star whose edges are subdivided an even number of times) and $2$-coronas of bipartite graphs discussed in Theorem \ref{The:2/3}, as well as some other trees like the ones presented in Figure~\ref{TightTrees}. It seems difficult to obtain a full characterization of this family (even for trees), given the diversity of these examples.

Note that the same bound as the one in Theorem~\ref{TheShift} has been proved for trees, in~\cite{CS11}, for locating-total dominating sets, and the trees reaching the bound are characterized therein. (A \emph{locating-total dominating set} is a set $D$ of vertices such that each vertex of $G$ has a neighbour in $D$, and any two vertices not in $D$ are separated by $D$. In other words, it is a locating-dominating set and total dominating set.) However, the extremal families are not the same: for example, the trees of Figure~\ref{TightTrees} have locating-total dominating sets smaller than the bound (examples of such sets are formed by the square vertices). Moreover, the upper bound for locating-total dominating sets cannot be generalized in the same way to bipartite graphs since the cycle $C_6$ requires at least four codewords in the case of locating-total domination. 

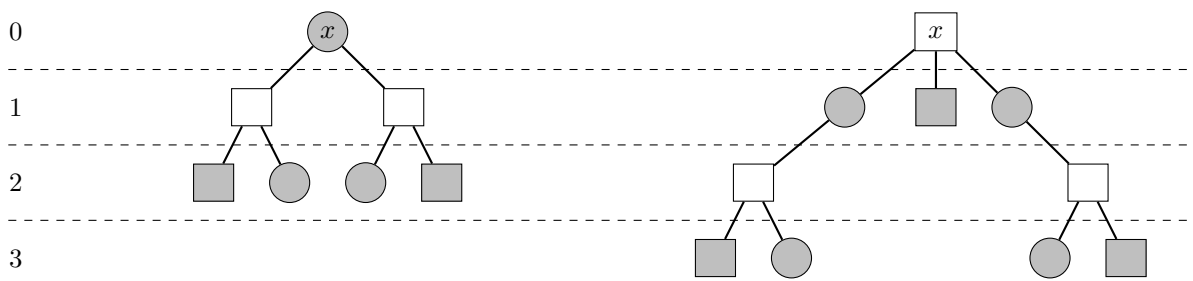
\begin{figure}[!htpb]
\centering
\begin{tikzpicture}

\draw[dashed] (-4.2,-0.5) -- (11.5,-0.5);
\draw[dashed] (-4.2,-1.5) -- (11.5,-1.5);
\draw[dashed] (-4.2,-2.5) -- (11.5,-2.5);

\node[code node](x) at (0,0)    {$ x $};
\node[square main node](1) at (-1,-1) {};

\node[square code node](11) at (-1.5,-2)    {};
\node[code node](12) at (-0.5,-2)  {};

\node[square main node](2) at (1,-1)  {};
\node[code node](21) at (0.5,-2)  {};
\node[square code node](22) at (1.5,-2)  {};

\node[] at (-4.1,-0) {$0$};
\node[] at (-4.1,-1) {$1$};
\node[] at (-4.1,-2) {$2$};
\node[] at (-4.1,-3) {$3$};

\path[draw,thick]
    (x) edge node {} (1)
    (1) edge node {} (11)
    (1) edge node {} (12)
    (x) edge node {} (2)
    (2) edge node {} (21)
    (2) edge node {} (22)

    ;
\node[square main node2](Rx) at (8,0)    {$ x $};
\node[code node](R1) at (6.8,-1) {};

\node[square main node](R11) at (5.6,-2)    {};
\node[square code node](R111) at (5.1,-3) {};
\node[code node](R112) at (6.1,-3) {};
\node[square code node](R0) at (8,-1)  {};

\node[code node](R2) at (9,-1)  {};
\node[square main node](R21) at (10,-2)  {};
\node[code node](R211) at (9.5,-3)  {};
\node[square code node](R212) at (10.5,-3)  {};

\path[draw,thick]
    (Rx) edge node {} (R1)
    (R1) edge node {} (R11)
    (Rx) edge node {} (R0)
    (R11) edge node {} (R111)
    (R11) edge node {} (R112)
    (Rx) edge node {} (R2)
    (R2) edge node {} (R21)
    (R21) edge node {} (R211)
    (R21) edge node {} (R212);
\end{tikzpicture}\centering
\caption{The gray vertices form optimal identifying codes in these two trees, whose sizes are equal to the bound presented in Theorem~\ref{TheShift}. The squared vertices form optimal locating-total dominating sets.}\label{TightTrees}
\end{figure}

We get the following corollary of Lemma~\ref{LemSupportBip} and Theorem~\ref{TheShift}.

\begin{corollary}\label{CorTreeBound}
Let $G$ be a connected bipartite graph on $n\geq5$ vertices which does not have any twins of degree~2 or greater. We have $$\ID(G)\leq\min\left\{\frac{n+\ell(G)}{2},n-s(G)\right\}.$$
\end{corollary}

\section{An application to twin-free bipartite graphs}\label{sec:twinfree}

We next apply our bounds to obtain upper bounds for twin-free graphs. Similar upper bounds have been studied in the context of location-domination, see~\cite{FH16,Heia,GGM14}.

\begin{corollary}\label{cor:2n/3}
Let $G$ be a twin-free bipartite graph on $n\geq3$ vertices that is not $P_4$. Then $$\ID(G)\leq\frac{2n}{3}.$$
\end{corollary}
\begin{proof}
Since $G$ is twin-free, we have $s(G)=\ell(G)$. Hence, Corollary~\ref{CorTreeBound} gives $\ID(G)\leq \min\{(n+\ell(G))/2,n-\ell(G)\}$. We first assume that $\ell(G)\leq n/3$. Then, $(n+\ell(G))/2\leq 2n/3$. Now, assume that $\ell(G)\geq n/3$. In that case, $n-\ell(G)\leq 2n/3$. Thus, the claim follows.\end{proof}

Note that the conditions in the statement of Corollary~\ref{cor:2n/3} are best possible, in the sense that there exist twin-free non-bipartite graphs that need $n-1$ vertices in any of their identifying codes, such as the complements of half-graphs, see~\cite{FGKNPV1l}. Those graphs have very large cliques, but there are twin-free graphs with small clique number that also have large identifying codes: for any $\Delta\geq 3$, arbitrarily large twin-free graphs of order $n$, maximum degree~$\Delta$ and optimal identifying codes of size $\frac{(\Delta-1)n}{\Delta}$ have been presented in~\cite{FP11}. The condition on twin-freeness cannot be relaxed either, as stars of order $n$ also have no identifying code smaller than $n-1$, and all other complete bipartite graphs (except $C_4$) need $n-2$ vertices in their identifying codes. Moreover, this bound does not hold for non-bipartite graphs of girth at least~5, since $\ID(C_7)=5$. 

We next show that the bound of Corollary~\ref{cor:2n/3} is tight.

\begin{proposition}\label{prop:coronas}
Let $H$ be any connected graph of order at least~$2$. Then, the 2-corona $H\circ_2$ is twin-free and $\ID(H\circ_2)=2n/3$, where $n$ is the order of $H\circ_2$.
\end{proposition}
\begin{proof}
For any vertex $v$ of $H$, let $v_1$ be the vertex adjacent to $v$ that was not in $H$, and let $v_2$ be the leaf adjacent to $v_1$. To separate $v_1$ from $v_2$, $v$ needs to belong to any identifying code of $H\circ_2$. Moreover, to dominate $v_2$, one of $v_1,v_2$ needs to belong to any identifying code. This shows that $\ID(H\circ_2)\geq 2n/3$.

For the upper bound, one can consider the set containing $v$ and $v_2$ for each vertex $v$ of $H$: this is an identifying code of $H\circ_2$. If $H$ has at least three vertices, the set containing $v$ and $v_1$ for each vertex $v$ of $H$ also works.
\end{proof}

In the following theorem, we characterize all the twin-free 
bipartite graphs achieving the upper bound of $2n/3$ of Corollary~\ref{cor:2n/3} by showing that they are exactly the $2$-coronas of bipartite graphs.

\begin{theorem}\label{The:2/3}
Let $G$ be a connected twin-free bipartite graph on $n$ vertices with $\ID(G)=2n/3$. Then $G$ is the 2-corona $H\circ_2$ of some bipartite graph $H$.
\end{theorem}
\begin{proof}
Since $G$ is twin-free, we have $s(G)=\ell(G)$. Together with Corollary \ref{CorTreeBound} and the fact that $\ID(G)=2n/3$, this means that $V(G)$ can be partitioned into three equal-sized parts: $$s(G)=\ell(G)=|V(G)\setminus(L(G)\cup S(G))|=n/3.$$ We will call the vertices in $V(G)\setminus(L(G)\cup S(G))$ \textit{central vertices}.

Let us first assume that $v$ is a central vertex without an adjacent support vertex. Notice that $v$ has at least two adjacent central vertices. Let us now construct code $C'=V(G)\setminus (L(G)\cup\{v\})$. Clearly, $v$ and each leaf are distinguished. Moreover, since $G$ is bipartite, all the codewords are distinguished (in other words, $G-\{v\}-L(G)$ is identifiable), unless there exists at least one 2-vertex component $w, w'$ in $G-\{v\}-L(G)$. One of $w,w'$ must be a central vertex adjacent to $v$, and the other, a support vertex. Assume that $w\in N(v)$ and let $w'$ be the support vertex in $G$. Now, we may just shift the codeword from $w'$ to the leaf adjacent to $w'$ in $G$. Moreover, we repeat this shifting for every 2-vertex component of $G-\{v\}-L(G)$. Since $v$ has at least two code neighbours, the code $C$ we get in this way is identifying in $G$ and $|C|=2n/3-1$, a contradiction. Hence, we may next assume that each central vertex is adjacent to a support vertex.

Let us now assume, that there exists a support vertex $u$ such that it has at least two adjacent central vertices $u_1$ and $u_2$. Then, we have $n/3$ central vertices but at most $n/3-1$ support vertices with exactly one adjacent central vertex. Hence, there exists a central vertex $v$ such that each support vertex adjacent to $v$ has also another adjacent central vertex as a neighbour.
Let us again choose $C'=V(G)\setminus (L(G)\cup\{v\})$. Notice that $|I(v)|\geq2$, thus $v$ and each leaf are distinguished. Moreover, we may apply the same argument on $G-\{v\}-L(G)$ as above and we may form $C$ in the same way.
Again, $C$ is identifying in $G$ and it has cardinality $2n/3-1$, a contradiction. 

Therefore, each support vertex is adjacent to exactly one central vertex and each central vertex is adjacent to exactly one support vertex. Finally we only need to show that no two support vertices are adjacent. If $u,w\in S(G)$, $d(u,w)=1$ and the central vertex adjacent to $u$ is $v$, then the code $G-\{v\}-L(G)$ is again identifying in $G$, unless there exists a component of size two in  $G-\{v\}-L(G)$ but in that case we can apply the earlier argument of shifting a codeword from a support vertex to a leaf to get an identifying code.
This means that $(G-L(G)-S(G))\circ_2=G$ and the claim follows.\end{proof}

\section{Graphs of girth at least~5}\label{sec:girth5}

In this section we prove our upper bound for graphs of girth at least~5, which generalizes Theorem~\ref{The:ID5/7delta}. It is natural to study these graphs here since bipartite graphs of girth at least $5$ are contained in Theorem~\ref{TheShift} and trees can be considered as graphs with unbounded girth. Notice that the new upper bound is tight for $C_7$ and stars. Moreover, it is the best possible bound, using only the order of graph $n$ and the number of leaves $\ell(G)$, in the sense that every non-leaf vertex has to increase the upper bound by $5/7$ as witnessed by the graph $C_7$ and each leaf has to increase the upper bound by $1$ as we have seen in the case of star graphs. Besides a new upper bound, we also present a new infinite family of graphs of girth at least $5$ which has large identifying codes.

\begin{theorem}\label{The:ID5/7}
Let $G$ be an identifiable graph of order $n$ with girth at least $5$ without isolated vertices. Then $\ID(G)\leq \frac{5n+2\ell(G)}{7}$.
\end{theorem}

\begin{proof}
It is sufficent to prove the claim for connected graphs as each connected component can be considered independently. Note that a graph of girth at least~5 is identifiable if and only if no connected component is a $P_2$. Thus, let $G$ be a connected identifiable graph on $n$ vertices of girth at least $5$ without isolated vertices.

We prove the claim by induction. Note that the upper bound can be written as $(n-\ell(G))5/7+\ell(G)$, which will be used in this way in the proof. Assume first that $n=3$ ($n\geq 3$ since $G$ is identifiable and is not an isolated vertex). In this case, $G$ is $P_3$ and we have $\ID(G)=2$. Let us then assume that the claim holds for all $n$ with $n\leq k$ and let us consider $n=k+1$. Observe that if $G$ has minimum degree at least~$2$ or if $G$ is bipartite, then we are done by Theorems~\ref{The:ID5/7delta} and~\ref{TheShift}. Hence, we may assume that there exists a vertex $v\in V(G)$ such that $v$ belongs to a cycle or there are two disjoint paths from $v$ to two cycles, and there is a cut-edge $vu$ such that $G-vu$ is disconnected and one of the components is a tree $T_u$ on $n_u\geq 1$ vertices which contains vertex $u$. We will perform a case analysis based on the structure of this tree and the surroundings of $v$. The basic idea is to apply Theorem~\ref{TheShift} on $T_u$ and use the induction hypothesis on $G_u=G-T_u$. Let us denote by $T_v$ the tree $G[V(T_u)\cup\{v\}]$, $G_v=G-T_v$ and let $C_u$  and $C_v$ be an optimal identifying code in $G_u$ and $G_v$, respectively. 

We will use the following observations throughout this proof. (i) If a non-codeword is dominated by two codewords, then it has a unique $I$-set, since $G$ has girth at least~5. Indeed, if any other vertex has those two codewords in its neighbourhood, then we have a 4-cycle. (ii) Similarly, we notice that if a codeword has three or more vertices in its $I$-set, then this $I$-set is unique.

We will next consider five cases.

\medskip

\noindent\textbf{Case 1.} Let us assume that $T_u=K_1$, that is, it is a single leaf-vertex $u$ attached to $v$. Moreover, let us assume that $v\not\in C_u$. We have one vertex less and one leaf less, thus, by the induction hypothesis, $|C_u|\leq (n-1-(\ell(G)-1))5/7+\ell(G)-1=(n-\ell(G))5/7+\ell(G)-1$ and hence, $C= C_u\cup\{u\}$ has cardinality at most $(n-\ell(G))5/7+\ell(G)$. Moreover, $C$ is an identifying code of $G$ since $I(C;u)=\{u\}$ and if $I(x)=\{u\}$, then $x$ is not dominated by $C_u$, a contradiction. If $v\in C_u$ and $I(C_u;v)\neq\{v\}$, then we may again consider code $C=C_u\cup\{u\}$. If $I(C_u;v)=\{v\}$, then $v$ has a non-codeword neighbour  $v'$ in $G_u$. Now, instead of $u$, we add $v'$ to code $C_u$, that is, $C=C_u\cup\{v'\}$. Observe that since $C_u$ is an identifying code in $G_u$, we have $|I(C;v')|\geq3$ in $G$ and thus, $u$ is the only vertex with $I(u)=\{v\}$ and thus, it is separated from all other vertices. All other vertex pairs are separated by the vertices of $C_u$.

\medskip

\noindent\textbf{Case 2.} Let us now assume that $T_u$ is a star $K_{1,n_u-1}$ where $n_u\geq3$. Let us assume first that $u$ is the central vertex of the star. In this case, $\ell(G)=\ell(G_u)+n_u-1$. Thus, we may just add all the $n_u-1$ leaves of $T_u$ to $C_u$ and the resulting code is an identifying code of the claimed cardinality. If $u$ is one of the leaves and $u'$ is the central vertex, then $\ell(G)=\ell(G_u)+n_u-2$. However, we also increase the number of non-leaf vertices by 2 when we transform $G_u$ into $G$ and hence, we may add $\lfloor n_u-2+10/7\rfloor$, that is $n_u-1=\ell(T_u)$, codewords to $C_u$. If $n_u\geq4$ ($T_u$ is not a path), then we may add to $C_u$ each vertex in $V(T_u)$ except $u$ to form an identifying code. If $n_u=3$, then $T_u$ is the path $P_3$. If $v\in C_u$, then we add vertices $u$ and $u'$ to $C_u$ and if $v\not\in C_u$, then we add the two leaves  of $T_u$ to $C_u$. In both cases we obtain an identifying code of the claimed size.

\medskip

\noindent\textbf{Case 3.} Let us now assume that $n_u-\ell(T_u)\geq2$ and $\ID(T_u)\leq\lfloor\frac{n_u-\ell(T_u)}{2}\rfloor+\ell(T_u)-1$. Let us denote by $C'_u$ the optimal identifying code in $T_u$. If $u\not\in C'_u$ or $v\not\in C_u$, then we may just consider $C=C_u\cup C'_u$ as the identifying code and we are done. If $v\in C_u$ and $u\in C'_u$, then we may be required to do some modifications to the code as we may have $I(v)=I(u)$. Since we have $\ell (G)\geq\ell(G_u)+\ell(T_u)-1$ depending on whether $u\in L(T_u)$, we may just add some codeword $u'$ adjacent to $u$ and we get an identifying code $C_u\cup C'_u\cup\{u'\}$ of claimed cardinality. Indeed, we have

\begin{align*}
&5\frac{n-\ell(G)}{7}+\ell(G)\\
\geq&5\frac{(n-n_u)-\ell(G_u)+n_u+1-\ell(T_u)}{7}+\ell(G_u)+\ell(T_u)-1\\
\geq&|C_u|+\frac{n_u-\ell(T_u)}{2}+\ell(T_u)-1+\frac{3(n_u-\ell(T_u))+10}{14}\\
>&|C_u|+|C_u'|+1.
\end{align*}

In the first inequality we use $\ell(G)\geq \ell(G_u)+\ell(T_u)-1$. We can do this since $5\frac{n-\ell(G)}{7}+\ell(G)=\frac{5n+2\ell(G)}{7}$. In the second inequality, we use the induction hypothesis with $|C_u|\leq 5\frac{n-n_u-\ell(G_u)}{7}+\ell(G_u)$. In the third inequality, we use our assumption $|C_u'|\leq \ID(T_u)\leq\lfloor\frac{n_u-\ell(T_u)}{2}\rfloor+\ell(T_u)-1$ and the assumption $n_u-\ell(T_u)\geq2$ to show that $\frac{3(n_u-\ell(T_u))+10}{14}>1.$ Hence, we may add the new codeword adjacent to $u$ and we are done.

\medskip

\noindent\textbf{Case 4.} Let us now assume that $n_u-\ell(T_u)\geq2$ and $\ID(T_u)=\lfloor\frac{n_u-\ell(T_u)}{2}\rfloor+\ell(T_u)$. Notice that in this case Theorem~\ref{TheShift} provides a tight bound and we may assume that identifying code $C'_u$ of $T_u$ has the structure provided by the proof. In particular, we may assume that for any non-codeword $w\in V(T_u)\setminus L(T_u)$, we have $|I(w)|\geq2$. Moreover, we notice that Theorem~\ref{TheShift} actually offers two identifying codes with cardinalities $\lfloor\frac{n_u-\ell(T_u)}{2}\rfloor+\ell(T_u)$ and $\lceil\frac{n_u-\ell(T_u)}{2}\rceil+\ell(T_u)$. Moreover, if a vertex $w$ is a non-codeword vertex in one these codes, then it is a codeword in the other one, and if $w$ is a codeword in both of them, then it has an adjacent codeword in at least one of these codes.

Let us assume first that $v\not\in C_u$ or $u\not\in C'_u$. Then we may consider the identifying code $C=C_u\cup C'_u$. Since $n_u-\ell(T_u)\geq2$, this code has the desired cardinality. Thus, we may assume from now on that $u\in C'_u$ and $v\in C_u$. We further split this case based on whether $u\in L(T_u)$ (and the parity of $n_u-\ell(T_u)$).

\medskip
\noindent\textbf{Subcase 4.1: \boldmath{$u\not\in L(T_u)$}.} In this case, if $n_u-\ell(T_u)$ is even, Theorem~\ref{TheShift} offers two identifying codes with equal cardinalities and we may choose the code in which $u$ is either a non-codeword or has adjacent codeword(s) and we are done. 

When $n_u-\ell(T_u)$ is odd, we have to do some calculations since the two codes have cardinalities $\frac{n_u-1-\ell(T_u)}{2}+\ell(T_u)$ and $\frac{n_u+1-\ell(T_u)}{2}+\ell(T_u)$, respectively. However, even the larger of these two codes is small enough. Indeed, $\frac{n_u+1-\ell(T_u)}{2}+\ell(T_u)+5\frac{n-n_u-\ell(G_u)}{7}+\ell(G_u)=5\frac{n-\ell(G)}{7}+\ell(G)+\frac{3(\ell(T_u)-n_u)+7}{14}.$ Since $n_u-\ell(T_u)$ is odd and at least 3, the last sum term is negative and hence, we may use either of the two codes also in this case. Notice that since $u\not\in L(T_u)$, we used $\ell(G)=\ell(G_u)+\ell(T_u)$.

\medskip
\noindent\textbf{Subcase 4.2: \boldmath{$u\in L(T_u)$}.} 
In this case, we have $\ell(G)=\ell(G_u)+\ell(T_u)-1$. Moreover, we have $\lfloor\frac{n_u-\ell(T_u)}{2}\rfloor+\ell(T_u)\leq \lfloor5\frac{n_u+1-\ell(T_u)}{7}\rfloor+\ell(T_u)-1$ when $n_u-\ell(T_u)\geq2$. Thus, code $C_u\cup C'_u$ has the claimed cardinality. Let $u'$ be the support vertex adjacent to $u$ in $T_u$. If $u'\in C'_u$, then we may use $C_u\cup C'_u$ as our identifying code. Hence, let us assume that $u'\not\in C'_u$. Notice that if $u'$ has an adjacent leaf other than $u$, then by the construction of $C'_u$, $|I(u')|\geq3$ and we can shift the codeword in $u$ to $u'$. In fact, each neighbour of $u'$ is a codeword in $C'_u$ since the structure of $C'_u$ is as in the proof of Theorem $6$ and hence, we may assume from now on that $\deg(u')=2$.

If $n_u-\ell(T_u)$ is even, then there is also another identifying code of the same cardinality as $C'_u$ in $T_u$ by the proof of Theorem~\ref{TheShift}. Moreover, this other code will have $u'$ as a codeword. 

If $n_u-\ell(T_u)$ is odd, consider the tree $T'=T_u-u$. We have $\ell(T_u)=\ell(T')$, since $\deg(u')=2$, and hence, $T'$ has an even number of non-leaves. If the optimal identifying code in $T'$ has cardinality $\frac{n_u-1-\ell(T_u)}{2}+\ell(T_u)$, then there exist two codes of this size and in one of the codes, let us say in $C_{T'}$, $u'$ is a codeword. Now $C_u\cup C_{T'}$ is an identifying code of $G$. On the other hand, if $T'$ has a smaller identifying code $C'_{T'}$, then $C_u\cup C'_{T'}\cup\{u'\}$ is an identifying code of $G$ of the correct size.

\medskip

\noindent\textbf{Case 5.} We are left with the case where $T_u$ is a path on two vertices. By applying the previous cases if possible, we may assume that each leaf of $G$ is part of a $P_2$ which is joined with a single edge to some vertex similar to $v$, belonging to a cycle or connected to at least two cycles. In particular, every support vertex has degree~2. Since $P_2$ is not identifiable, $C'_u$ does not exist. Let $u'$ be the leaf neighbour of $u$ in $T_u$. 
Observe that if $v\in C_u$, then we can consider the code $C_u\cup\{u'\}$, which is identifying and of the correct size. Thus, we assume that $v\not\in C_u$. In particular, if multiple $P_2$'s are connected to $v$, $v$ is forced to be a codeword in $C_u$ to separate a leaf and its support vertex, and we are done. Hence, we may assume that there is at most one support vertex adjacent to $v$ in $G$. This implies that $G_v$ is identifiable. We further split this case based on the number of leaves in $G_v$.

If we have $\ell(G_v)<\ell(G)$, that is, $\ell(G_v)=\ell(G)-1$, then we may consider the graphs $G_v$ and the path on three vertices formed by $v,u$ and $u'$. If $C_v\cap N(v)=\emptyset$, then we may consider code $C=C_v\cup \{v,u'\}$ and if $C_v\cap N(v)\neq\emptyset$, then we may consider code $C=C_v\cup \{v,u\}$. Both of these codes have the claimed cardinality and they are identifying codes.

For the rest of the proof we may assume that $\ell(G_v)=\ell(G)$. In other words, exactly one vertex in $N_{G}(v)$ is a leaf in $G_v$, that is, it has degree~2 in $G$. Let us denote this vertex by $v_1$ and by $v'_1$ the second neighbour of $v_1$. If $\deg(v'_1)\geq3$, then $\ell(G_v-v_1)=\ell(G_v)-1=\ell(G)-1$. Let $C_{v_1}$ be an optimal identifying code in $G_v-v_1$ (note that this graph is identifiable). By induction we have $|C_{v_1}|\leq 5(n-4-\ell(G_v)+1)/7+\ell(G_v)-1=5(n-3-\ell(G))/7+\ell(G)-1$. We may now consider the code $C=C_{v_1}\cup \{v,u,u'\}$. This code has cardinality at most $5(n-3-\ell(G))/7+\ell(G)-1+3< 5(n-\ell(G))/7+\ell(G)$ and is an identifying code. Thus, we may assume that $\deg(v'_1)=2$. Moreover, let us denote by $C_{v'_1}$ the optimal identifying code in $G_{v'_1}=G_v-v_1-v'_1$ (note that this graph is also identifiable). Observe that $\ell(G_{v'_1})\leq\ell(G_v)$. We now consider the code $C=C_{v'_1}\cup\{u,v,v_1\}$, which is identifying. It has cardinality $|C|\leq 5(n-5-\ell(G_{v'_1}))/7+\ell(G_{v'_1})+3\leq 5(n-\ell(G_{v}))/7+\ell(G_{v})=5(n-\ell(G))/7+\ell(G)$.

For the rest of the proof we may assume that in $G_v$ we have leaves $v_1$ and $v_2$ which are of degree~2 and adjacent to $v$ in $G$. 
We know that $\deg(v)\geq3$ in $G$. Let us construct tree $T^*_v$ by starting from $T_v$ and adding to $T_v$ vertex $v_1$, and iteratively, any vertex of degree~2 adjacent to the previous vertex. We do this until we reach a vertex (denoted by $w$) that does not have degree~2. Let us denote by $x_t$ the last vertex we add to $T^*_v$ in this way. 
It is possible that $w=v$, if we have a suitable cycle in $G$. In that case, if $t>1$, we do not include the final edge between $v$ and $x_t$. We denote by $P_t$ the path from $v_1$ to the leaf $x_t$ in $T^*_v$ that is adjacent to $w$ in $G$. Let $P_t$ be the path on vertices $x_1,x_2,\dots, x_t$ where consecutive vertices are adjacent and $x_1=v_1$, and let us denote graph $G_P=G-P_t$ with an optimal identifying code $C_P$ (note that this graph is identifiable). Notice that $|V(G_P)|\geq3$ since it at least  contains vertices $v, u$ and $u'$. Moreover, if $|V(G_P)|=3$, then $G$ consists of a cycle, leaf and a support vertex. Now we may choose some optimal identifying code for the cycle so that $v$ is a codeword which gives us a case which we have already considered. Hence,  $|V(G_P)|>3$ and $\ell(G_P)=\ell(G)$.

\medskip

Let us further split this case into five subcases, based on the value of $t$. 

\medskip

\noindent\textbf{Subcase 5.1: \boldmath{$t\geq 5$}.} 
Recall that $\ID(P_t)=\lceil (t+1)/2\rceil$ since $t\geq5$,~\cite{BCHL04}. Notice that  $v\in C_P$ as it is the only vertex which can separate $u$ and $u'$. Assume first that $t$ is even. In that case $|C_P|+\ID(P_t)\leq 5(n-t-\ell(G_P))/7+\ell(G_P)+t/2+1=5(n-\ell(G))/7+\ell(G)+(14-3t)/14$. We will use code $\{x_2,x_4,\dots,x_{t-2},x_{t-1},x_t\}$ for the path. Together with $C_P$ it has the claimed cardinality and is identifying in $G$ (notice that $I(x_1)=\{v,x_2\}$).

Let us then consider the case where $t\geq5$ is odd. As above, we get that $|C_P|+\ID(P_t)\leq 5(n-t-\ell(G_P))/7+\ell(G_P)+(t+1)/2=5(n-\ell(G))/7+\ell(G)+(7-3t)/14$. We may use set $\{x_2,x_4,\dots,x_{t-3},x_{t-2},x_{t-1}\}$ together with $C_P$.  Therefore, we may now assume that $t\leq4$. 

\medskip

\noindent\textbf{Subcase 5.2: \boldmath{$t=4$}.} When $t=4$, we can add two codewords to $C_P$ to form a code in $G$ since $4\cdot5/7\geq2$ and $\ell(G)=\ell(G_P)$. Recall that $v\in C_P$ and that $x_1=v_1$. If $w\not\in C_P$, we add $x_2,x_4$ on $P_t$ to $C_P$. If $w\in C_P$, then, if $u'\in C_P$, we shift it to $u$ and after that add codewords $x_1,x_3$. Notice that to dominate $u'$ at least one of $u$ and $u'$ is a codeword of $C_P$.

\medskip

\noindent\textbf{Subcase 5.3: \boldmath{$t=3$}.} When $t=3$, we can again add two codewords since $3\cdot 5/7\geq2$. The codewords we add are $x_1$ and $x_2$.

\medskip

\noindent\textbf{Subcase 5.4: \boldmath{$t=2$}.} When $t=2$, we can add at most one codeword. Again we shift the codeword possibly in $u'$ to $u$ and after that add vertex $x_1$ as a codeword. 

\medskip

\noindent\textbf{Subcase 5.5: \boldmath{$t=1$}.} 
Consider the graph $G^*=G-v_1-u'$ (which is identifiable) with optimal identifying code $C^*$. We have $\ell(G)=\ell(G^*)$ since $v_1$ has degree~2 and has no adjacent vertices of degree~2. Hence, if we can construct an identifying code $C$ for graph $G$ by adding at most one codeword to $C^*$, then $C$ has the claimed cardinality. Observe that at least one of $u$ and $v$ belongs to code $C^*$ to dominate $u$. Moreover, we have $|I(v)|\geq2$ as one vertex is needed to separate $u,v$.

Let us assume first that $v\not\in C^*$ and $u\in C^*$. In that case, we can consider code $C^*\cup\{v\}$ for $G$. Now $I(u')=\{u\}$ and if $I(v_1)=I(z)$, then $v\in I(z)$ and hence, $|I(z)|\geq2$ since $z$ was dominated by $C^*$. Since $z$ is $2$-dominated, $v_1$ and $z$ are separated. Thus, $C$ is an identifying code.

Consider then the case with $v,u\in C^*$. In this case, we consider code $C=C^*\cup\{v_1\}$. Clearly $I(u')$ is unique  and $v_1$ and $v$ are the only vertices with $v_1$ and $v$ in their $I$-sets. Since $u\in I(v)$, also $I(v_1)$ is unique. Hence, $C$ is an identifying code.

Finally, we are left with the case $v\in C^*$ and $u\not \in C^*$. We consider the code $C=C^*\cup\{u\}$. Again $I(u')$ is unique. Moreover, in $G^*$ vertex $u$ is the only vertex whose $I$-set is $\{v\}$. Thus, if $I(v_1)=\{v\}$, then $v_1$ is separated from every other vertex and if $|I(v_1)|\geq2$, then $I(v_1)$ is clearly unique. Thus, $C$ is an identifying code of claimed cardinality in $G$. As this was final case, we have now proven the claim.\end{proof}

In~\cite{BFH15}, the authors have constructed an infinite family of connected graphs without leaves which have girth at least~5 and $\ID(G)=3(n-1)/5$, where $n$ is the order. To date, this infinite family of graphs features the largest known ratio between $\ID(G)$ and the number of vertices, among graphs without leaves and with girth at least $5$ (apart from some small examples such as the $7$-cycle). The interest to these kinds of constructions is due to the fact that Theorem~\ref{The:ID5/7delta} is tight only for the $7$-cycle and so, perhaps there exists a way to improve the bound for connected graphs by excluding the $7$-cycle as a single exception. New constructions which increase the ratio $\ID(G)/n$ give new limits to how much the bound of Theorem~\ref{The:ID5/7delta} could possibly be improved. In the following proposition, we give a new infinite family of such graphs which offers the largest known ratio for $\ID(G)/n$ for graphs of girth at least $5$ with no leaves.

\begin{figure}[h]
\centering
\begin{tikzpicture}

\draw[thick, dotted] (-0.5,6) -- (3.5,6);
\draw[thick, dotted] (0.5,4) -- (2.5,4);

\node[main node](x) at (0,8)    {};

\node[code node](1) at (-4,7)    {$x_1$};
\node[code node](2) at (-1,7)    {$x_2$};
\node[code node](3) at (4,7)    {$x_k$};

\node[code node](10) at (-4,6)    {$v_1$};
\node[code node](11) at (-3,5)    {};
\node[main node](12) at (-3,4)    {};
\node[code node](13) at (-3,3)    {};

\node[main node](16) at (-5,5)    {};
\node[code node](15) at (-5,4)    {};
\node[main node](14) at (-5,3)    {};

\node[code node](20) at (-1,6)    {$v_2$};
\node[code node](21) at (0,5)    {};
\node[main node](22) at (0,4)    {};
\node[code node](23) at (0,3)    {};

\node[main node](26) at (-2,5)    {};
\node[code node](25) at (-2,4)    {};
\node[main node](24) at (-2,3)    {};

\node[code node](30) at (4,6)    {$v_k$};
\node[code node](31) at (3,5)    {};
\node[main node](32) at (3,4)    {};
\node[code node](33) at (3,3)    {};

\node[main node](36) at (5,5)    {};
\node[code node](35) at (5,4)    {};
\node[main node](34) at (5,3)    {};

\path[draw,thick]
    (x) edge node {} (1)
    (x) edge node {} (2)
    (x) edge node {} (3)
    (1) edge node {} (10)
    (10) edge node {} (11)
    (11) edge node {} (12)
    (12) edge node {} (13)
    (13) edge node {} (14)
    (14) edge node {} (15)
    (15) edge node {} (16)
    (16) edge node {} (10)
    
    (2) edge node {} (20)
    (20) edge node {} (21)
    (21) edge node {} (22)
    (22) edge node {} (23)
    (23) edge node {} (24)
    (24) edge node {} (25)
    (25) edge node {} (26)
    (26) edge node {} (20)
    
    (3) edge node {} (30)
    (30) edge node {} (31)
    (31) edge node {} (32)
    (32) edge node {} (33)
    (33) edge node {} (34)
    (34) edge node {} (35)
    (35) edge node {} (36)
    (36) edge node {} (30)

    ;

\end{tikzpicture}\centering
\caption{Graph $G$ of girth $7$ with no leaves, on $8k+1$ vertices with $\ID(G)=5k$. Gray vertices form an optimal identifying code.}\label{5/8Example}
\end{figure}
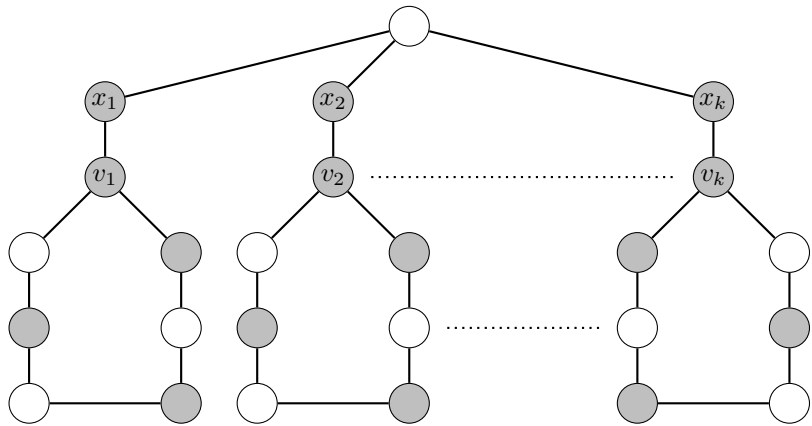

\begin{proposition}\label{Prop5/8}
For each integer $k\geq1$, there exists a connected graph $G$ on $n=8k+1$ vertices with $\ID(G)=5k=\frac{5(n-1)}{8}$.
\end{proposition}
\begin{proof}
Let $k\geq1$ be an integer and let us construct graph $G$ by taking a star $K_{1,k}$ and by attaching a unique $7$-cycle to each leaf of the star with a single edge. The resulting graph has $k+1+7k=8k+1$ vertices. Let us denote the leaves of $K_{1,k}$ by $x_1,\dots, x_k$ and vertices in the $7$-cycles adjacent to them by $v_1,\dots, v_k$. Graph $G$ is illustrated in Figure \ref{5/8Example}.

Let us now consider the identifying code number of $G$. Let $C$ be an optimal identifying code. We claim that there are at least five code vertices among the eight vertices in $x_i$ and the $7$-cycle attached to it.
Recall that we have $\ID(C_7)=5$. Thus, if $x_i\not\in C$, then we have at least five code vertices in the $7$-cycle. Assume then that $x_i\in C$ and that there exists a set $C'$ of three vertices in the $7$-cycle such that $C'\cup\{x_i\}$ dominates and distinguishes all of these vertices. If $I(v_i)=\{x_i\}$, then $C'$ is an identifying code for $C_7-v_i$. However, $C_7-v_i$ is a $6$-path $P_6$ and we have $\ID(P_6)=4$, a contradiction. Hence, $C'$ is a dominating set of $C_7$. Since $C'\cup\{x_i\}$ distinguishes every vertex in $C_7$ and $\ID(C_7)=5$, we have $N[v_i]\cap C'=N[w]\cap C'$ for some $w\in V(C_7)$ and there are no other vertices with the same $I$-sets. Let $u$ be a vertex in $N[w]\setminus N[v_i]$. Now $C'\cup \{u\}$ is an identifying code of size $4$ in $C_7$, a contradiction. Thus, $\ID(G)\geq 5k$.

Finally, we show that $\ID(G)\leq 5k$. To construct an identifying code $C$, we choose vertices $x_i$ and $v_i$ for each $i$, and for each vertex $v_i$ we choose one adjacent vertex $w_i$ in the cycle and then we choose two additional code vertices $u_1$ and $u_2$ in each cycle so that $I(u_i)=\{u_i\}$. This code is depicted with gray vertices in Figure~\ref{5/8Example}. One can easily check that the code is indeed identifying.
\end{proof}

\section{Concluding remarks}\label{sec:conclu}

We have improved several bounds from the literature, both in terms of the values of the bounds, and/or in terms of the generality of the considered graph classes. Our bounds confirm the known facts that certain structural graph features such as leaves, twins or short cycles are crucial for a graph to have a large identifying code. By considering the number of leaves on graphs other than trees, our bounds enable us to quantify the effect of these structures on the identifying code number.

In Section \ref{SecBounds}, we have given a new tight upper bound $\ID(G)\leq (n+\ell(G))/2$ for bipartite graphs without twins of degree two or greater. We have characterized all twin-free graphs attaining this bound. However, it would be interesting to see a characterization for all graphs attaining the bound.

Our bound $\ID(G)\leq\frac{5n+2\ell(G)}{7}$ from Theorem~\ref{The:ID5/7} for graphs of girth at least~5 is tight for stars, the path $P_4$ and the 7-cycle. However, we do not know of any other tight examples. Perhaps this bound can be extended by considering other structural properties of the graph, to give a bound that is tight for a more diverse class of graphs. Perhaps it can also be improved by excluding the 7-cycle as an exception?  We have shown in Proposition~\ref{Prop5/8} that there are arbitrarily large connected twin-free graphs of girth~7 without leaves with $\ID(G)=5(n-1)/8$, hence such an improved bound could not be less than that. 

More generally, it would be interesting to see whether other bounds can be proved for graphs of larger girth. For example, perhaps a stronger version of the bound $\ID(G)\leq\frac{5n+2\ell(G)}{7}$ (for graphs of girth at least $5$) of Theorem~\ref{The:ID5/7} can be proved for graphs of girth at least~$g$ with $g\geq 9$. As we have $\ID(C_g)=(g+1)/2+1$ for odd $g$, such an upper bound cannot be less than that.

\end{document}